\documentclass[12pt,a4paper]{amsart}

\usepackage{fullpage}
\usepackage[pdftex]{epsfig}

\usepackage{amssymb}  % defines \nmid, for example
\usepackage{latexsym} % for \Box
\usepackage{comment}
\usepackage{color}
\usepackage{enumerate}
\usepackage[all]{xy}

\usepackage{charter,euler} % for better looks
\DeclareSymbolFont{bchoperators}{T1}{bch}{m}{n}
\SetSymbolFont{bchoperators}{bold}{T1}{bch}{b}{n}
\makeatletter
\renewcommand{\operator@font}{\mathgroup\symbchoperators}
\makeatother

\usepackage{titlesec} % change small caps in section titles to boldface
\titleformat{\section}{\normalfont\bfseries\filcenter}{\thesection}{1em}{}
\titleformat{\subsection}{\normalfont\bfseries}{\thesubsection}{1em}{}
\titleformat{\subsubsection}{\normalfont\bfseries}{\thesubsubsection}{1em}{}

\newcommand{\Z}{{\mathbb Z}}
\newcommand{\Q}{{\mathbb Q}}
\newcommand{\R}{{\mathbb R}}
\newcommand{\F}{{\mathbb F}}
\newcommand{\C}{{\mathbb C}}

\newcommand{\CO}{{\mathcal O}}

\newcommand{\bfomega}{{\underline \omega}}
\newcommand{\bfa}{{\underline a}}

\newcommand{\Char}{\operatorname{char}}
\newcommand{\disc}{\operatorname{disc}}

\newcommand{\im}{\operatorname{im}}

\newcommand{\GL}{\operatorname{GL}}

\newcommand{\Sel}{\operatorname{Sel}}

\newcommand{\Tr}{\operatorname{Tr}}

\newcommand{\inj}{\hookrightarrow}

\newcommand{\odd}{{\text{\rm odd}}}
\newcommand{\tors}{{\mathrm{tors}}}

\newcommand{\pws}[1]{[\![#1]\!]}

\newcommand{\lbar}{\overline{\log'}}

   {\begin{list}{}{\settowidth{\labelwidth}{#1}%
                   \setlength{\leftmargin}{\labelwidth}%
                   \addtolength{\leftmargin}{\labelsep}}}%
   {\end{list}}

\newtheorem{Theorem}{Theorem}[section]
\newtheorem{Lemma}[Theorem]{Lemma}
\newtheorem{Proposition}[Theorem]{Proposition}
\newtheorem{Corollary}[Theorem]{Corollary}

\theoremstyle{definition}

\newtheorem{Remark}[Theorem]{Remark}

\newtheorem{Assumption}[Theorem]{Assumption}
\newtheorem{Assumptions}[Theorem]{Assumptions}

\numberwithin{equation}{section}

\definecolor{darkgreen}{rgb}{0,0.5,0}
\definecolor{rem}{rgb}{0.8,0,0}
\definecolor{new}{rgb}{0.7,0,0.6}
\definecolor{reply}{rgb}{0,0,0.8}
\usepackage[
        colorlinks, citecolor=darkgreen,
        backref,
        pdfauthor={Michael Stoll}, % add other authors
        pdftitle={An application of ``Selmer group Chabauty'' to arithmetic dynamics},
]{hyperref}
\usepackage[alphabetic,backrefs,lite]{amsrefs} % for bibliography

\newcounter{ass}

\begin{document}

\title[An application of ``Selmer group Chabauty'' to arithmetic dynamics]%
      {An application of ``Selmer group Chabauty'' \\ to arithmetic dynamics}
\author{Michael Stoll}
\address{Mathematisches Institut,
         Universit\"at Bayreuth,
         95440 Bayreuth, Germany.}
\email{Michael.Stoll@uni-bayreuth.de}

\date{\today}


\begin{abstract}
  We describe how one can use the ``Selmer group Chabauty'' method developed by the author
  to show that certain hyperelliptic curves of the form
  \[ C \colon y^2 = x^{2g+1} + h(x)^2 \,, \]
  with $h \in \Z[x]$ such that $\deg h \le g$ and $h(0)$ is~odd,
  have only the ``obvious'' rational points~$\infty$ (the unique point at infinity
  on the smooth projective model of the curve) and $(0, \pm h(0))$.

  Let $c \in \Q$ and write $f_c(x) = x^2 + c$. We denote the iterates of $f_c$
  by $f_c^{\circ n}$; i.e., we set $f_c^{\circ 0}(x) = x$ and
  $f_c^{\circ(n+1)}(x) = f_c(f_c^{\circ n}(x))$.
  As an application of the method, we prove that if $f_c^{\circ 2}$ is irreducible,
  then $f_c^{\circ 6}$ is also irreducible. Assuming the Generalized Riemann Hypothesis (GRH),
  it also follows that $f_c^{\circ 10}$ is irreducible.
\end{abstract}

\maketitle

%\renewcommand{\baselinestretch}{1.1}
%\renewcommand{\arraystretch}{1.3}

%=============================================================================

\section{Introduction} % \label{}

We describe how one can use the ``Selmer group Chabauty'' method
developed by the author in~\cite{Stoll2017b} to show that certain
hyperelliptic curves of the form
\begin{equation} \label{E:curve}
  C \colon y^2 = x^N + h(x)^2 \,,
\end{equation}
where $N = 2g + 1$ is odd, $h \in \Z[x]$ with $\deg h \le g$ and $h(0)$~odd, have only
the ``obvious'' rational points $\infty$ (the unique point at infinity on the
smooth projective model of the curve) and $P_0 = (0, h(0))$, $(0, -h(0))$.

We note that curves of the form~\eqref{E:curve} can be characterized by a
nice property.

\begin{Lemma} \label{L:order}
  Let $C \colon y^2 = f(x)$ be a hyperelliptic curve of genus~$g$ over a field~$k$
  with $\Char(k) \neq 2$ such that $f$ is monic of odd degree~$2g+1$.
  We use the unique point~$\infty$ at infinity to define an embedding
  $i \colon C \to J$, where $J$ is the Jacobian variety of~$C$.
  \begin{enumerate}[\upshape(1)]\addtolength{\itemsep}{6pt}
    \item If $P \in C(k) \setminus \{\infty\}$ is such that $i(P) \in J(k)$
          has finite order~$n$, then $n = 2$ or $n \ge 2g+1$.
    \item If $P \in C(k)$ is such that $i(P)$ has order~$2g+1$ and $x(P) = 0$, then
          $f = x^{2g+1} + h(x)^2$ with a polynomial~$h \in k[x]$ of degree~$\le g$
          such that $P = (0, h(0))$. Conversely, if $f$ is of this form and
          $P = (0, h(0))$, then $i(P)$ has order~$2g+1$.
  \end{enumerate}
\end{Lemma}

\begin{proof} \strut
  \begin{enumerate}[(1)]\addtolength{\itemsep}{6pt}
    \item The divisor $n \cdot (P - \infty)$ is principal; let $\phi$ denote
          a function in~$C(k)^\times$ with this divisor. Then $\phi \in L(n \cdot \infty)$.
          If $n \le 2g$, then this Riemann-Roch space is generated by powers
          of~$x$. The divisor of a function of~$x$ alone is invariant under the
          hyperelliptic involution. This forces $P$ to be a Weierstrass point on~$C$,
          but then $i(P)$ has order~$2$.
    \item Taking $\phi$ as in~(1), we now have $\phi \in L((2g+1) \cdot \infty)$,
          which has basis $1, x, \ldots, x^g, y$. As we have just seen, $\phi$ cannot
          be a function of~$x$ alone, so (up to scaling), \hbox{$\phi = y - h(x)$}
          with $h \in k[x]$ of degree~$\le g$. The fact that $P = (0, h(0))$
          is a zero of order~$2g+1$ of~$\phi$ implies that $x^{2g+1}$ divides
          $f(x) - h(x)^2$. Since $f$ is monic of degree $2g+1$ (and $\deg h \le g$),
          it follows that $f(x) = x^{2g+1} + h(x)^2$. For the converse, we check
          that the divisor of $y - h(x)$ is $(2g+1) \cdot (P - \infty)$. Since
          $P \neq \infty$ and $2g+1$ is odd, the claim then follows from~(1).
    \qedhere
  \end{enumerate}
\end{proof}

As an application of the method, we prove the following result.

\begin{Theorem} \label{T:appl}
  Let $c \in \Q$ and write $f_c(x) = x^2 + c$. We denote the iterates of~$f_c$
  by~$f_c^{\circ n}$; i.e., we set $f_c^{\circ 0}(x) = x$
  and $f_c^{\circ(n+1)}(x) = f_c(f_c^{\circ n}(x))$.
  If $f_c^{\circ 2}$ is irreducible, then $f_c^{\circ 6}$ is also irreducible.
  Assuming the Generalized Riemann Hypothesis (GRH), it also follows that
  $f_c^{\circ 10}$ is irreducible.
\end{Theorem}

This is based on the following observation. Define the polynomial
$A_n(c) = f_c^{\circ n}(0)$, for $n \ge 1$; i.e., $A_1(c) = c$ and
$A_{n+1}(c) = A_n(c)^2 + c$. The following implication holds when $n \ge 2$
(see~\cite{Stoll1992}*{Lemma~1.2}).
\[ \text{$f_c^{\circ(n-1)}$ irreducible and $A_n(c)$ not a square} \;\Longrightarrow
   \text{$f_c^{\circ n}$ irreducible}
\]
See also~\cite{Jones2013} for a survey on related questions.
By results in~\cite{DHJMS}, we know that $A_3(c)$ and~$A_4(c)$ are never squares
unless $c = 0$, or $c = -1$ in the case of~$A_4$, but then $f_c = x^2$ or $x^2 - 1$
is already reducible.
So the theorem follows if we can show that both $A_5(c)$ and~$A_6(c)$
(and $A_7(c)$, $A_8(c)$, $A_9(c)$, $A_{10}(c)$, where we will need to
assume~GRH for~$A_7$) can never be a square
for $c \neq 0, -1$. We note that $\deg A_n = 2^{n-1}$ and that $A_n(0) = 0$.
The curve $y^2 = A_n(x)$, whose rational points are of interest to us,
is isomorphic to the curve $y^2 = a_n(x)$, where
\[ a_n(x) = x^{2^{n-1}} A_n(x^{-1}) \,, \]
so
\[ a_1(x) = 1 \qquad\text{and}\qquad a_{n+1}(x) = x^{2^n-1} + a_n(x)^2 \,. \]
This shows that the curves under consideration have the form studied in this
paper. For $n = 5$, we have the explicit equation
\begin{align*}
  y^2 = a_5(x) &=
    x^{15} + x^{14} + 2 x^{13} + 5 x^{12} + 14 x^{11} + 26 x^{10} + 44 x^9 + 69 x^8 \\
    & \qquad {} + 94 x^7 + 114 x^6 + 116 x^5 + 94 x^4 + 60 x^3 + 28 x^2 + 8 x + 1 \,.
\end{align*}
Our method shows that this curve
has only the three obvious rational points $\infty$, $(0,1)$ and~$(0,-1)$,
which is equivalent to saying that $A_5(c)$ can never be a square for $c \neq 0$.

We also show that $A_6(c)$ is a square only for $c = 0$ and $c = -1$
using by now standard methods for rational points on curves of genus~$2$.

We show that $A_7(c)$ can only be a square when $c = 0$ by computing the
$2$-Selmer group of the Jacobian of the curve $y^2 = a_7(x)$; this Selmer
group turns out to be trivial. We need to assume GRH to verify that the
class group of the number field obtained by adjoining a root of~$a_7$ to~$\Q$
is trivial.

We can reduce the question whether $A_n(c)$ can be a square for a given $c \in \Q$
to the analogous question for $A_p(c)$, $\pm A_p(c)$, or $A_4(c)$, where
$p$ is an odd prime divisor of~$n$, for $n$ odd, $n$ twice an odd number,
and $n$ divisible by~$4$, respectively. Thus we can deal with $A_8$ and~$A_9$
by reducing the question to $A_4$ and~$A_3$, respectively. We show,
again using Selmer group Chabauty, that
$-A_5(c)$ can be a square only for $c = 0, -1$, which, together with
our result on~$A_5$, also deals with~$A_{10}$. We note that it is only
the implication
\[ \text{$f^{\circ 6}_c$ irreducible} \;\Longrightarrow\; \text{$f^{\circ 7}_c$ irreducible} \]
that depends on~GRH, whereas the implication
\[ \text{$f^{\circ 7}_c$ irreducible} \;\Longrightarrow\; \text{$f^{\circ 10}_c$ irreducible} \]
is unconditional, since it follows from the unconditional results
on $A_8$, $A_9$, and~$A_{10}$.

We describe the method in detail in Section~\ref{S:method}, prove the claims
on the various $A_n(c)$ in Section~\ref{S:dyn}, and give
some further examples and statistics in Section~\ref{S:examples}.

%=============================================================================

\section{Selmer group Chabauty for a family of hyperelliptic curves} \label{S:method}

Fix $g \ge 1$. We consider the hyperelliptic curve
\[ C \colon y^2 = f(x) := x^{2g+1} + h(x)^2 \,, \]
where $h \in \Z[x]$, and we make the following assumptions (we will make
additional assumptions later).

\begin{samepage}
\begin{Assumptions} \label{A:Ass1} \strut
  \begin{enumerate}[({A}1)]\addtolength{\itemsep}{3pt}
    \item \label{ass1} $\deg h \le g$, so $f$ is monic of degree~$2g+1$.
    \item \label{Ass2} $h(0)$ is odd.
    \setcounter{ass}{\theenumi}
  \end{enumerate}
\end{Assumptions}
\end{samepage}

By considering $f$ and its derivative $f'(x) = (2g+1) x^{2g} + 2 h(x) h'(x)$
modulo~$2$, we see that $f$ is squarefree over~$\F_2$, hence
$f$ is squarefree and the discriminant of~$f$ is odd,
and so $C$ is a hyperelliptic curve of genus~$g$. Since $\deg f = 2g+1$ is odd,
there is a unique (and hence rational) point at infinity on (the smooth projective
model of) $C$ with respect to the given affine equation. We denote this point
by~$\infty$. We write~$J$ for the Jacobian variety of~$C$; we have the embedding
$i \colon C \to J$, $P \mapsto [P - \infty]$, defined over~$\Q$.
We note that $C$ has at least the three obvious rational points $\infty$ and~$(0, \pm h(0))$.
We write $P_0 = (0, h(0))$; if $\iota \colon C \to C$, $(x,y) \mapsto (x,-y)$,
denotes the hyperelliptic involution on~$C$, then the third point is~$\iota(P_0)$.

\begin{Lemma} \label{L:potgoodred2}
  $C$ has potentially good reduction at~$2$.
  More precisely, let $\pi = 2^{1/(2g+1)}$; then $C$ is isomorphic to
  \[ C' \colon \eta^2 + h(\pi^2 \xi) \eta = \xi^{2g+1} \]
  over~$\Q(\pi)$, with reduction mod~$\pi$ given by
  \[ \tilde{C} \colon \eta^2 + \eta = \xi^{2g+1} \,, \]
  which is a hyperelliptic curve of genus~$g$ over~$\F_2$.
\end{Lemma}

\begin{proof}
  We set
  \[ y = 2 \eta + h(\pi^2 \xi) \qquad\text{and}\qquad x = \pi^2 \xi \]
  in the equation defining~$C$; some elementary manipulations then
  give us the equation of~$C'$. Since $h(0)$ is odd by Assumption~(A\ref{Ass2}),
  we have that $h(\pi^2 \xi) \equiv 1 \bmod \pi$, and so we get the indicated
  reduction~$\tilde{C}$ mod~$\pi$. The affine part of the reduction is smooth, since the
  partial derivative of the equation with respect to~$\eta$ is constant~$= 1$,
  so $\tilde{C}$ is a curve of genus~$g$. Since the genus does not drop,
  $C'$ has good reduction at the prime above~$2$ of~$\Q(\pi)$.
\end{proof}

Let $C(\Q)_\odd$ denote the set of points $P \in C(\Q)$
such that $i(P)$ has odd finite order; similarly for~$C(\Q_2)_\odd$.

\begin{Lemma} \label{L:odd}
  $C(\Q)_\odd = C(\Q_2)_\odd = \{\infty, P_0, \iota(P_0)\}$.
\end{Lemma}

\begin{proof}
  By Lemma~\ref{L:potgoodred2}, we have the following commutative diagram
  \[ \SelectTips{eu}{12}
     \xymatrix{ C(\Q) \ar@{^(->}[r] & C(\Q_2) \ar@{^(->}[r] \ar[dr]^-\alpha
                     & C(\Q_2(\pi)) \ar@{^(->}[r]^-i \ar[d] & J(\Q_2(\pi)) \ar[d] \\
                   & & \tilde{C}(\F_2) \ar@{^(->}[r]^-i & \tilde{J}(\F_2)\rlap{ \,,}
              }
  \]
  where $\tilde{J}$ denotes the Jacobian of~$\tilde{C}$.
  The right-most vertical map is injective on the torsion points
  of odd order. This implies that $\alpha \colon C(\Q_2)_\odd \to \tilde{C}(\F_2)$
  is also injective. Since $\tilde{C}(\F_2) = \{\infty, (0,0), (0,1)\}$
  consists of three points, this already shows that $\#C(\Q_2)_\odd$
  has at most three elements.
  On the other hand, $i(\infty) = 0$ has order~$1$ and
  $i(P_0)$ and $i(\iota(P_0)) = -i(P_0)$ have odd order~$2g+1$ (see Lemma~\ref{L:order}).
  So
  \[ \{\infty, P_0, \iota(P_0)\} \subset C(\Q)_\odd \subset C(\Q_2)_\odd
     \qquad\text{and}\qquad \#C(\Q_2)_\odd \le 3 \,,
  \]
  which proves the claim.
\end{proof}

We note that $\tilde{J}[2] = 0$ (the hyperelliptic involution on~$\tilde{C}$
is given by $\eta \mapsto \eta + 1$; the point at infinity is the only fixed point),
so that $J(\Q_2)[2]$ maps into the kernel of reduction in~$J(\Q_2(\pi))$.

For later use, we now consider regular differentials on~$C$ and their
formal integrals over~$\Q_2$. We write $v_2 \colon \bar{\Q}_2 \to \Q \cup \{\infty\}$
for the $2$-adic valuation normalized such that $v_2(2) = 1$.

\begin{Lemma} \label{L:diff}
  We set, for $j = 0,1,\ldots,g-1$, $\omega_j = x^j\,dx/y \in \Omega^1(C)$.
  Then the following holds.
  \begin{enumerate}[\upshape(1)]\addtolength{\itemsep}{6pt}
    \item \label{L:diff:2}
          Let $t = x$ be a uniformizer at~$P_0$ and write
          $\omega_j = w_j(t)\,dt$ with power series $w_j \in \Q_2\pws{t}$.
          Let $\ell_j(t) = \int_0^t w_j(\tau)\,d\tau \in \Q_2\pws{t}$ be
          the formal integral. Then the coefficient of~$t^n$ in~$\ell_j$
          has $2$-adic valuation at least $\frac{2}{2g+1} -\frac{2}{2g+1} n - \log_2 n$.
          In particular, the series $\ell_j(t)$ converge when $v_2(t) > \frac{2}{2g+1}$.
    \item \label{L:diff:3}
          Let $\tilde{t} = (y - h(x))/x^{g+1}$; this is a uniformizer at~$\infty$.
          Write $\omega_j = \tilde{w}_j(\tilde{t})\,d\tilde{t}$ with $\tilde{w}_j \in \Q_2\pws{\tilde{t}}$
          and let
          $\tilde{\ell}_j(\tilde{t}) = \int_0^{\tilde{t}} \tilde{w}_j(\tau)\,d\tau \in \Q_2\pws{\tilde{t}}$.
          Then the coefficient of~$\tilde{t}^n$ in~$\tilde{\ell}_j$ has
          $2$-adic valuation at least $\frac{2}{2g+1} + \frac{1}{2g+1} n - \log_2 n$.
          In particular, the series $\tilde{\ell}_j(\tilde{t})$ converge
          when $v_2(\tilde{t}) > -\frac{1}{2g+1}$.
  \end{enumerate}
\end{Lemma}

\begin{proof}
  Since $C'$ has good reduction, the differentials $\omega'_j = \xi^j\,d\xi/(2\eta + h(\pi^2\xi))$,
  for $j = 0,1,\ldots,g-1$, form a basis of the N\'eron differentials on~$C'$
  over~$\Z_2[\pi]$. In terms of the isomorphism between $C$ and~$C'$
  given in the proof of Lemma~\ref{L:potgoodred2}, we find that
  $\omega_j = \pi^{2j+2} \omega'_j$.

  To prove~\eqref{L:diff:2}, we observe that the~$\omega'_j$ reduce to
  regular differentials on~$\tilde{C}$, which implies that we can write
  $\omega'_j(\tau) = u_j(\tau)\,d\tau$ with $u_j \in \Z_2[\pi]\pws{\tau}$, where
  $\tau = \xi = \pi^{-2} t$ is a uniformizer at the image of~$P_0$ on~$C'$ that
  reduces to a uniformizer at the reduction of this point on~$\tilde{C}$. Then
  \[ w_j(t)\,dt = \omega_j = \pi^{2j+2} \omega'_j = \pi^{2j+2} u_j(\tau)\,d\tau
                = \pi^{2j} u_j(\pi^{-2}t)\,dt \,,
  \]
  so the coefficient of~$t^n$ in~$w_j$ has valuation
  $\ge (2j - 2n) v_2(\pi) \ge -\frac{2}{2g+1} n$. The claim follows.

  The proof of~\eqref{L:diff:3} is similar. $\tilde{\tau} = \eta/\xi^{g+1} = \pi\tilde{t}$
  is a uniformizer at~$\infty$ on~$C'$ that reduces to a uniformizer
  at~$\infty$ on~$\tilde{C}$. As before, we have that
  $\omega'_j(\tilde{\tau}) = \tilde{u}_j(\tilde{\tau})\,d\tilde{\tau}$
  with $\tilde{u}_j \in \Z_2[\pi]\pws{\tilde{\tau}}$. Then
  \[ \tilde{w}_j(\tilde{t})\,d\tilde{t} = \omega_j = \pi^{2j+2} \omega'_j
              = \pi^{2j+2} \tilde{u}_j(\tilde{\tau})\,d\tilde{\tau}
              = \pi^{2j+3} \tilde{u}_j(\pi\tilde{t})\,d\tilde{t} \,,
  \]
  so the coefficient of~$\tilde{t}^n$ in~$\tilde{w}_j$ has valuation
  $\ge (2j + 3 + n) v_2(\pi) \ge \frac{1}{2g+1}(3 + n)$, and the claim follows.
\end{proof}

One can show that $(\omega_0/2, \ldots, \omega_{g-1}/2)$
is a basis of the $\Z_2$-module of N\'eron differentials on~$C$ over~$\Z_2$,
but we will not need this. We set $\bfomega = (\omega_0, \ldots, \omega_{g-1})$.

Recall that pull-back under~$i$ induces a canonical isomorphism between
the space of regular (equivalently, invariant) $1$-forms on~$J$ and
the space of regular differentials on~$C$. We write $\bfomega_J$ for
the basis of $\Omega^1(J_{\Q_2})$ corresponding under this isomorphism
to~$\bfomega$. On the compact commutative $2$-adic Lie group~$J(\Q_2)$ we have
the logarithm homomorphism $J(\Q_2) \to T_0 J(\Q_2)$, where the target
is the tangent space at the origin. Since the space of invariant differentials
on~$J_{\Q_2}$ is canonically isomorphic to the cotangent space at the origin,
the logarithm corresponds to a pairing between~$J(\Q_2)$ and the
space of invariant differentials. Using the isomorphism mentioned above,
the pairing is given by
\[ J(\Q_2) \times \Omega^1(C_{\Q_2}) \ni \bigl(\sum_{j=1}^m i(P_j), \omega\bigr)
    \longmapsto \sum_{j=1}^m \log_\omega P_j := \sum_{j=1}^m \int_\infty^{P_j} \omega \,,
\]
where the $P_j$ are points in~$C(\bar{\Q}_2)$ and
the integral here is the Berkovich-Coleman integral on~$C$ over~$\Q_2$;
see~\cite{Berkovich2007}.
(It agrees with the abelian integral, since the Berkovich space
of~$C$ over~$\C_2$ is simply connected; this comes from the fact
that $C$ has potentially good reduction.)
We define $\log \colon J(\Q_2) \to \Q_2^g$ as the $g$-tuple of homomorphisms
given in this way by pairing with~$\bfomega$. Then $\log$ induces
a $\Z_2$-linear map on the kernel of reduction in~$J(\Q_2)$.
In particular, its image~$\Lambda$ is a full $\Z_2$-lattice.
We fix an isomorphism $\Lambda \simeq \Z_2^g$ of $\Z_2$-modules and
define $\log' \colon J(\Q_2) \to \Z_2^g$ to be the composition of~$\log$
with this isomorphism. This induces an $\F_2$-linear surjective map
$\lbar \colon J(\Q_2)/2J(\Q_2) \to \F_2^g$. Its kernel
is the image of~$J(\Q_2)[2^\infty]$ in~$J(\Q_2)/2J(\Q_2)$ (the kernel
of~$\log'$ is the same as the kernel of~$\log$, which is~$J(\Q_2)_{\tors}$,
and torsion of odd order is mapped to zero in the quotient).

To make use of $\lbar$ later, we will need to know $\Lambda = \log(J(\Q_2))$.
We note that Lemma~\ref{L:diff} shows that when $P = (x,y)$
is a point on~$C$ with $v_2(x) < \frac{2}{2g+1}$,
then its image~$P'$ on~$C'$ has $v_2(\xi) = v_2(x) - \frac{2}{2g+1} < 0$,
so it reduces to~$\infty$ on~$\tilde{C}$, and we have that
$\tilde{t}(P) = \pi^{-1} \tilde{\tau}(P')$ has valuation $> -\frac{1}{2g+1}$.
We can then evaluate $\log i(P) = \int_\infty^P \bfomega \in \Q_2^g$
numerically by evaluating the series $\tilde{\ell}_j$ at~$\tilde{t}$
to any desired $2$-adic accuracy. Similarly, if $v_2(x) > \frac{2}{2g+1}$,
we can evaluate $\log i(P)$ numerically by evaluating the series
$\ell_j$ at~$x$, since we have that $\log i(P) = \log i(P) - \log i(P_0) = \int_{P_0}^P \bfomega$
($i(P_0)$ has finite order and so is killed by the logarithm).
We recall that every point $Q \in J(\Q_2)$ is represented by a divisor
of the form $D - d \cdot \infty$, where $D$ is an effective divisor
of degree $d \le g$ defined over~$\Q_2$.
This implies that $J(\Q_2)$ is generated by points of the form
$\Tr_{K/\Q_2} i(P)$ where $P \in C(K)$ and $K$ is an extension of~$\Q_2$
of degree~$\le g$. Since the ramification index of~$K$ is at most~$g$,
it follows that $v_2(x(P)) \neq \frac{2}{2g+1}$, so that we can
evaluate $\log i(P)$ for all relevant~$P$ in the way sketched above.

Write $A = \Q[x]/\langle f \rangle$; we denote
the image of~$x$ in~$A$ by~$\theta$, so that $A = \Q[\theta]$. Then
$A$ is an \'etale algebra over~$\Q$; it is a product of algebraic
number fields corresponding to the irreducible factors of~$f$.
Recall that the $2$-Selmer group of~$J$ can be identified with a
subgroup of~$A^\times/A^{\times 2}$; more precisely, a subgroup
of elements of~$A(S, 2)$ with trivial norm in~$\Q^\times/\Q^{\times 2}$.
Here $S$ is the set of primes of~$A$ above prime numbers~$p$
such that $p^2$ divides the discriminant of~$f$, and $A(S, 2)$
is the (finite) subgroup of~$A^\times/A^{\times 2}$ consisting
of elements whose valuation at all primes outside~$S$ is even;
see~\cites{Schaefer1995,Stoll2001}.
We let $H$ denote the subgroup of~$A(S, 2)$ consisting of elements
with trivial norm. We identify the $2$-Selmer group~$\Sel_2(J)$
with its image in~$H$.
(Usually, we would have to include the primes above~$2$ in~$S$,
but here we can leave them out, since $f$ has odd discriminant.
The standard argument then shows that elements in the image of the
Selmer group must have even valuation at the primes above~$2$.
See also the proof of Lemma~\ref{L:imlog}~\eqref{L:imlog2}, which shows
that enough images of points land in~$H_2$ to span the full image.)

We have the $2$-adic local analogue
\[ H_2 = \Z_2[\theta]^\times/\Z_2[\theta]^{\times 2} \]
(note that $\Z_2[\theta]$ is the ring of integers of
$A_2 = \Q_2[x]/\langle f \rangle = A \otimes_{\Q} \Q_2$)
and the canonical homomorphism $\rho_2 \colon H \to H_2$.
Write $\delta_2$ for the map $J(\Q_2) \to H_2$ given by the ``$x-T$ map''
on~$J(\Q_2)$. It lands in~$H_2$ by the parenthesized remark above.
We note that by definition, $\rho_2$ maps the $2$-Selmer group
into the image of~$J(\Q_2)/2J(\Q_2)$ in~$H_2$ under the local descent map.
So we can define (slightly abusing notation) the $\F_2$-linear map
$\lbar \circ \rho_2 \colon \Sel_2(J) \to \F_2^g$.

\begin{Remark}
  We note that we can use knowledge of the image~$J_2$ of~$J(\Q_2)/2J(\Q_2)$
  in~$H_2$ to determine~$\Sel_2(J)$ from the ``relaxed'' $2$-Selmer group~$\widetilde{\Sel}$
  of~$J$, where we leave out the local condition at~$2$. This is because
  $\Sel_2(J) = \widetilde{\Sel} \cap \rho_2^{-1}(J_2)$. This can give a faster
  way of determining the Selmer group compared to the standard implementation
  in~Magma. See also~\cite{Stoll2017b}*{Remark~4.2}.
\end{Remark}

We write $\delta$ for the global descent map $J(\Q) \to \Sel_2(J) \subseteq H$.
Its kernel is~$2J(\Q)$, and we have $\delta_2 = \rho_2 \circ \delta$ on~$J(\Q)$.
We need a further assumption.

\begin{samepage}
\begin{Assumption} \label{A:Ass2} \strut
  \begin{enumerate}[({A}1)]\addtolength{\itemsep}{3pt}\setcounter{enumi}{\theass}
    \item \label{ass3} The map $\lbar \circ \rho_2 \colon \Sel_2(J) \to \F_2^g$ has
          kernel~$\delta\bigl(J(\Q)[2^\infty]\bigr)$.
    \setcounter{ass}{\theenumi}
  \end{enumerate}
\end{Assumption}
\end{samepage}

Since we will exclusively deal with vector spaces over~$\F_2$ in the following,
all dimensions will be $\F_2$-dimensions.

Note that $\delta\bigl(J(\Q)[2^\infty]\bigr)$ is always contained in the kernel
of $\lbar \circ \rho_2$.
Since we have $\dim \delta\bigl(J(\Q)[2^\infty]\bigr) = \dim J(\Q)[2]$, the assumption
is equivalent to
\[ \dim \Sel_2(J) - \dim J(\Q)[2] = \dim \lbar\bigl(\rho_2(\Sel_2(J))\bigr) \,. \]
The term on the left is the upper bound on the rank of~$J(\Q)$ that we obtain
from the $2$-Selmer group. In the applications, we need to compute~$\lbar$
on~$J(\Q_2)/2J(\Q_2)$ in any case, so the assumption is easy to check.

\begin{Remark}
  If $J(\Q)[2] = 0$ (which is equivalent to $f$ being irreducible over~$\Q$),
  then the assumption reduces to saying that $\lbar \circ \rho_2$
  is injective, which in particular implies that $\rho_2$ is injective
  on the Selmer group, but the condition is in general stronger than this.
\end{Remark}

Define $\rho \colon \Z_2^g \setminus \{0\} \to \F_2^g \setminus \{0\}$,
$\bfa \mapsto \overline{2^{-v_2(\bfa)} \bfa}$.

\begin{Proposition} \label{P:key}
  Assume~{\upshape(A\ref{ass3})} and let $Q \in J(\Q)$ with $\log' Q \ne 0$
  (i.e., $Q \notin J(\Q)_\tors$).
  Then $\rho \log' Q \in \lbar\bigl(\rho_2(\Sel_2(J))\bigr)$.
\end{Proposition}

\begin{proof}
  Since $\log'\bigl(2^k J(\Q) + J(\Q)_\tors\bigr) \subset 2^k \Z_2^g$ and $\log' Q \ne 0$,
  there is a largest $k \ge 0$ with $Q \in 2^k J(\Q) + J(\Q)_\tors$.
  Write $Q = 2^k Z + T$ with $T \in J(\Q)_\tors$; then $Z \notin 2J(\Q) + J(\Q)_\tors$.
  So $\delta(Z) \in \Sel_2(J)$ is not in $\delta(J(\Q)_\tors) = \delta(J(\Q)[2^\infty])$
  (odd torsion lies in $2J(\Q)$), and by~(A\ref{ass3}), $\lbar(\delta_2(Z)) \ne 0$, that is,
  $\log' Z \notin 2\Z_2^g$. Hence
  \[ \rho \log' Q = \rho(2^k \log' Z) = \log' Z \bmod 2
                     = \lbar\bigl(\rho_2(\delta(Z))\bigr)  \in \lbar\bigl(\rho_2(\Sel_2(J))\bigr) \,.
     \qedhere
  \]
\end{proof}

This has the following consequence, which we will use in our applications.
It is the essence of the ``Selmer group Chabauty'' method from~\cite{Stoll2017b}
for the determination of the rational points on~$C$.

\begin{Corollary} \label{C:crit}
  Assume~{\upshape(A\ref{ass3})}. If $\rho \log' i(Q) \notin \lbar\bigl(\rho_2(\Sel_2(J))\bigr)$
  for every $Q \in C(\Q_2)$ such that $i(Q) \notin J(\Q_2)_\tors$, then $i(P) \in J(\Q)_\tors$
  for every $P \in C(\Q)$. In particular, if in addition $f$ is irreducible,
  then $C(\Q) = C(\Q)_\odd$.
\end{Corollary}

\begin{proof}
  Let $P \in C(\Q)$. If $\log' i(P) = 0$, then $i(P) \in J(\Q)_\tors$ as claimed
  (and if $f$ is irreducible, every torsion point has odd order, so $P \in C(\Q)_\odd$).
  Otherwise, Proposition~\ref{P:key} with $Q = i(P)$ gives a contradiction.
\end{proof}

Here is a variant that is more convenient when there are other rational points
(we will need it for the proof of Proposition~\ref{P:A_10}).

\begin{Corollary} \label{C:disc}
  Assume~{\upshape(A\ref{ass3})}. Let $B \in C(\Q)$ and $X \subseteq C(\Q_2)$ be such that
  $\log' [P - B] \neq 0$ and $\rho(\log'[P - B]) \notin \lbar\bigl(\rho_2(\Sel_2(J))\bigr)$
  for all $P \in X \setminus \{B\}$. Then $X \cap C(\Q) \subseteq \{B\}$.
\end{Corollary}

\begin{proof}
  Apply Proposition~\ref{P:key} to $Q = [P - B]$ for $P \in X \cap C(\Q)$ with
  $P \neq B$.
\end{proof}

We now come back to determining the image of~$\log$ for the family of curves
considered in this paper.
By Nakayama's Lemma, $\log(J(\Q_2))$ is generated as a $\Z_2$-module
by the images of representatives of a basis of (the $\F_2$-vector space)
$J(\Q_2)/2J(\Q_2)$. We know that $\dim J(\Q_2)/2J(\Q_2) = g + \dim J(\Q_2)[2]$,
and $\dim J(\Q_2)[2]$ is one less than the number of irreducible
factors of~$f$ over~$\Q_2$; see~\cite{Stoll2001}. Also, we have the
injective homomorphism $\delta_2 \colon J(\Q_2)/2J(\Q_2) \to H_2$,
which we can compute fairly easily.

A more thorough analysis allows us to pin down $\log(J(\Q_2))$ precisely
without having to search for suitable points in~$J(\Q_2)$.
First we need a handle on~$H_2$.

\begin{Lemma} \label{L:H2}
  We denote reduction mod~$2$ by a bar.
  Note that $\F_2[\bar{\theta}] = F_1 \times \cdots \times F_m$
  is a direct product of finite field extensions of~$\F_2$
  (this is because the discriminant of~$f$ is odd).
  For $\alpha \in \Z_2[\theta]$, we let $\bar{\alpha}_j \in F_j$ denote the $j$th
  component of $\bar{\alpha} \in \F_2[\bar{\theta}]$.
  Let $R \subset \Z_2[\theta]$ be a complete set of representatives of~$\F_2[\bar{\theta}]$.
  \begin{enumerate}[\upshape(1)]\addtolength{\itemsep}{6pt}
    \item Let $\alpha \in \Z_2[\theta]$. Then $1 + 4\alpha$ is a square in~$\Z_2[\theta]$
          if and only if $\Tr_{F_j/\F_2} \bar{\alpha}_j = 0$ for all~$j$. We define
          $\Tr \colon \F_2[\bar{\theta}] \to \F_2^m$ as
          $\Tr(\bar{\alpha}) = (\Tr_{F_j/\F_2} \bar{\alpha}_j)_{j=1,\ldots,m}$.
    \item Let $B \subset R$ be representatives of an $\F_2$-basis of~$\F_2[\bar{\theta}]$
          and let $B' \subset R$ be representatives of an $\F_2$-basis of
          $\F_2[\bar{\theta}]/\ker \Tr$. Then $(1 + 2\beta)_{\beta \in B}$
          together with $(1 + 4\beta')_{\beta' \in B'}$
          is an $\F_2$-basis of~$H_2$.
  \end{enumerate}
\end{Lemma}

\begin{proof}
  Let $\Pi = 1 + 2\Z_2[\theta]$ be the group of principal units. Since
  $\Z_2[\theta]^\times/\Pi \simeq F_1^\times \times \cdots \times F_m^\times$
  is finite of odd order, the inclusion $\Pi \inj \Z_2[\theta]^\times$ induces
  an isomorphism $\Pi/\Pi^2 \simeq H_2$.
  \begin{enumerate}[(1)]\addtolength{\itemsep}{6pt}
    \item If $1 + 4\alpha$ is a square, it must be the square of an element
          of the form $1 + 2\beta$ with \hbox{$\beta \in \Z_2[\theta]$}.
          Now $(1 + 2\beta)^2 = 1 + 4(\beta + \beta^2)$,
          so the condition is that \hbox{$\alpha = \beta + \beta^2$} for some
          \hbox{$\beta \in \Z_2[\theta]$}. Since $\Tr \bar{\beta} = \Tr \bar{\beta}^2$,
          every such~$\alpha$ satisfies $\Tr \bar{\alpha} = 0$.
          Conversely, if \hbox{$\Tr \bar{\alpha} = 0$}, there is $\bar{\beta} \in \F_2[\bar{\theta}]$
          such that $\bar{\alpha} = \bar{\beta} + \bar{\beta}^2$.
          Then $1 + 4\alpha \equiv (1 + 2\beta)^2 \bmod 8\Z_2[\theta]$,
          which implies that $1 + 4\alpha$ is a square as well.
    \item We clearly have that $\Pi^2 \subset 1 + 4\Z_2[\theta]$.
          The result of~(1) implies that $1 + 4\alpha \mapsto \Tr \bar{\alpha}$ induces
          an isomorphism $(1 + 4\Z_2[\theta])/\Pi^2 \simeq \F_2^m$.
          In particular, $(1 + 4\beta')_{\beta' \in B'}$ is a basis
          of the $\F_2$-vector space on the left.
          Also, $1 + 2\beta \mapsto \bar{\beta}$ induces an isomorphism
          $\Pi/(1 + 4\Z_2[\theta]) \simeq \F_2[\bar{\theta}]$, so that
          $(1 + 2\beta)_{\beta \in B}$ gives a basis of the space on the left.
          Combining these two bases gives the result.
    \qedhere
  \end{enumerate}
\end{proof}

Now we can use this information to get at the image of~$\log$.

\begin{Lemma} \label{L:imlog}
  We write ${\Sigma\!\Tr} \bar{\alpha}$ for $\sum_{j=1}^m \Tr_{F_j/\F_2} \bar{\alpha}_j \in \F_2$,
  with notation as in Lemma~\ref{L:H2}.
  \begin{enumerate}[\upshape(1)]\addtolength{\itemsep}{6pt}
    \item \label{L:imlog1}
          We have that $\Sigma\!\Tr \bar{\theta}^{-d} = 0$ for $1 \le d \le 2g$.
    \item \label{L:imlog2}
          The image in~$H_2$ of~$J(\Q_2)$ under~$\delta_2$ has basis represented by
          $(1 - 2 \theta^{-d})_{d=1,\ldots,g}$ and $(1 + 4 \theta^{-d})_{d \in I}$,
          where $I \subset \{1,3,5,\ldots,2g-1\}$ is such that $(\bar{\theta}^{-d})_{d \in I}$
          is a basis of $\ker {\Sigma\!\Tr}/\ker \Tr$.
    \item \label{L:imlog3}
          Let $\bfomega = \sum_{n=0}^\infty t^n \bfa_n\,dt$ with $t = x$ a uniformizer
          at~$P_0 = (0, h(0))$ and $\bfa_n \in \Q_2^g$ with $v_2(\bfa_n) \ge \lceil -\frac{2n}{2g+1} \rceil$.
          Then
          \[ \log(J(\Q_2)) = \bigl\langle \sum_{n=1}^\infty \frac{2^n}{n} \bfa_{dn-1}
                                          : d = 1,\ldots,g \bigr\rangle_{\Z_2}
                             + \bigl\langle \sum_{n=1}^\infty \frac{4^n}{n} \bfa_{dn-1}
                                            : d \in I \bigr\rangle_{\Z_2} \,.
          \]
  \end{enumerate}
\end{Lemma}

\begin{proof}
  To prove~\eqref{L:imlog1}, recall that if $p$ is the reciprocal characteristic
  polynomial of $\alpha \in L$, where $L$ is an \'etale algebra over a field~$K$,
  then we have the following identity of formal power series over~$K$.
  \[ \sum_{n=1}^\infty (\Tr_{L/K} \alpha^n) z^n = \frac{-z p'(z)}{p(z)} \,. \]
  (We have that $p(z) = z^{[L:K]} c(z^{-1}) = 1 - (\Tr_{L/K} \alpha) z + \ldots$
  when $c$ is the characteristic polynomial of~$\alpha$.)
  We apply this to $\alpha = \bar{\theta}^{-1} \in \F_2[\bar{\theta}] = L$
  and $K = \F_2$. In this case,
  $p(z) = \bar{f}(z) = \bar{h}(z)^2 + z^{2g+1} \in \F_2[z]$, so that
  \[ \frac{-z p'(z)}{p(z)} = \frac{z (2 \bar{h}(z) \bar{h}'(z) + (2g+1) z^{2g})}{\bar{f}(z)}
                           = \frac{z^{2g+1}}{\bar{f}(z)} \in \F_2\pws{z} \,,
  \]
  which shows that
  $\Sigma\!\Tr \bar{\theta}^{-d} = \Tr_{\F_2[\bar{\theta}]/\F_2} \bar{\theta}^{-d} = 0$
  when $1 \le d \le 2g$.

  We know that $\dim J(\Q_2)/2J(\Q_2) = \dim J(\Q_2)[2] + g$.
  Since
  \[ \#I = \dim (\ker {\Sigma\!\Tr}/\ker \Tr) = m-1 = \dim J(\Q_2)[2] \,, \]
  to show~\eqref{L:imlog2} it is enough to show that the given elements are in the image
  of~$\delta_2$, since their number is correct and they are linearly independent
  according to~Lemma~\ref{L:H2}. (Note that the reductions of $\theta^{-d}$,
  $0 \le d \le 2g$, form a basis of~$\F_2[\bar{\theta}]$, so by~\eqref{L:imlog1},
  we can select representatives
  of a basis of $\ker {\Sigma\!\Tr}/\ker \Tr$ from powers of~$\theta$.
  Also note that $\Sigma\!\Tr 1 = 1$ and $\Tr \bar{\theta}^{-2d} = \Tr \bar{\theta}^{-d}$,
  so it is sufficient to consider powers with odd exponents.)

  Let $K$ be a finite extension of~$\Q_2$ and take $\alpha \in K$ with $v_2(\alpha) > \frac{2}{2g+1}$.
  Then
  \[ f(\alpha) = \alpha^{2g+1} + h(\alpha)^2 \equiv h(\alpha)^2 \bmod 4\pi_K \,, \]
  where $\pi_K$ is a uniformizer of~$K$, so $f(\alpha)$ is a square in~$K$
  (note that $h(\alpha)$ is a unit), and there is a point $P = (\alpha, \beta) \in C(K)$.
  Then $\Tr_{K/\Q_2} i(P) \in J(\Q_2)$, and if $c$ is the characteristic polynomial
  of~$\alpha$, then $\delta_2(\Tr_{K/\Q_2} i(P)) = (-1)^{\deg c} c(\theta)$
  (we specify elements of~$H_2$ by representatives in $\Z_2[\theta]^\times$).

  Now consider $\alpha = 2^{1/d} \in \Q_2(2^{1/d})$ for $1 \le d \le g$.
  Then $v_2(\alpha) = \frac{1}{d} > \frac{2}{2g+1}$, so we obtain a point $Q_d \in J(\Q_2)$
  with
  \[ \delta_2(Q_d) = (-1)^d (\theta^d - 2) = (-\theta)^d (1 - 2 \theta^{-d}) \,. \]
  Since $-\theta = (\theta^{-g} h(\theta))^2$ is a square in~$\Z_2[\theta]$, this
  shows that $1 - 2 \theta^{-d}$ is in the image of~$\delta_2$. Similarly,
  taking $\alpha = 4^{1/d} \in \Q_2(2^{1/d})$ with $d$ odd and $d \le 2g-1$,
  we again have $v_2(\alpha) > \frac{2}{2g+1}$, and we find that $1 + 4 \theta^{-d}$
  is in the image of~$\delta_2$. This proves~\eqref{L:imlog2}.

  To see~\eqref{L:imlog3}, we recall that the image of~$\log$ is the $\Z_2$-span of the
  logarithms of elements of~$J(\Q_2)$ that give generators of~$J(\Q_2)/2J(\Q_2)$.
  By the reasoning above, such elements are given by divisors of the form
  (the $y$-coordinates~$*$ are chosen so that the point lies in the residue disk of~$P_0$)
  \[ \sum_{k=0}^{d-1} (\zeta_d^k 2^{1/d}, *) - d \cdot \infty\,, \quad 1 \le d \le g, \qquad\text{and}\qquad
     \sum_{k=0}^{d-1} (\zeta_d^k 4^{1/d}, *) - d \cdot \infty\,, \quad d \in I \,.
  \]
  Here $\zeta_d$ denotes a primitive $d$th root of unity in~$\bar{\Q}_2$.
  The claim follows by noting that
  \[ \log [(\alpha, *) - \infty] = \sum_{n=1}^\infty \frac{\alpha^n}{n} \bfa_{n-1} \]
  when $v_2(\alpha) > \frac{2}{2g+1}$ and that
  $\Tr_{\Q_2(2^{1/d})/\Q_2} (2^{k/d}) = 0$ when $d \nmid k$ and $= d 2^{k/d}$ otherwise.
  (We use that $d$ is odd when $d \in I$, so that $x^d - 4$ is irreducible over~$\Q_2$.)
\end{proof}

Note that in~\eqref{L:imlog3} we can replace $I$ by any larger set contained
in~$\{1,3,\ldots,2g-1\}$, since we get the logarithm of some element of~$J(\Q_2)$
for any such~$d$. So to determine the image of~$\log$, we first compute
the $(2g \times g)$-matrix~$L$ whose rows are given by the individual logarithms specified in the
lemma (computed to sufficient precision), but with $I = \{1,3,5,\ldots,2g-1\}$.
Then we echelonize this matrix over~$\Z_2$, which gives us a $\Z_2$-basis
of the image. If the precision turns out to be insufficient for the echelonization
process, we have to recompute the logarithms to higher precision.
We obtain a matrix $U \in \GL(g, \Z[\frac{1}{2}])$ such that the rows of~$L U$
generate~$\Z_2^g$.
Then $\log' \colon Q \mapsto (\log Q) U$ is a logarithm on~$J(\Q_2)$
with image~$\Z_2^g$ (we consider $\log Q$ as a row vector here), which is what
we need.

\begin{Remark}
  Lemma~\ref{L:imlog} shows, among other things, that $J(\Q_2)/2J(\Q_2)$ can be
  generated by classes represented by divisors $[D - (\deg D) \cdot \infty]$
  with $D$ reducing to a multiple of~$\bar{P}_0$. This can be seen independently
  as follows. Recall that $\tilde{J}(\F_2)$ has odd order. Let $Q \in J(\Q_2)$
  be arbitrary. Then an odd multiple of~$Q$ is in the kernel of reduction
  in~$J'(\Q_2(\pi))$ (where $J'$ is the Jacobian of $C'$ as in Lemma~\ref{L:potgoodred2}),
  hence we can assume that $Q$ is in this kernel of reduction
  to begin with. We can represent $Q$ by a divisor of the form $D - (\deg D) \cdot P_0$
  with $\deg D \le g$. Since $\bar{P}_0$ is not a Weierstrass point, the reduction
  of~$D$, considered as a divisor on~$C'$, must be $(\deg D) \cdot \bar{P}_0$ (because
  $Q$ is in the kernel of reduction). We can add $(\deg D) \cdot [P_0 - \infty]$
  without changing the image mod~$2J(\Q_2)$, since $[P_0 - \infty]$ is a torsion
  point of odd order.
\end{Remark}

In the following, we will outline an algorithm that can be used
to verify the assumption in Corollary~\ref{C:crit} for a given curve~$C$
and thus enable us to determine~$C(\Q)$.
Recall that we first need to verify Assumption~(A\ref{ass3}).
This can be done once we can compute~$\lbar$ on a basis of~$J(\Q_2)/2J(\Q_2)$.
Then we need to show that $\rho \log' i(P) \notin \lbar\bigl(\rho_2(\Sel_2(J))\bigr)$
for every $P \in C(\Q_2)$ such that $i(P) \notin J(\Q_2)_\tors$.

We first consider points $P \in C(\Q_2)$ such that $x(P) \in \Z_2^\times$.

\begin{Lemma} \label{L:v=0}
  Assume that $h(1)$ is even. Then
  all points $P \in C(\Q_2)$ such that $x(P) \in \Z_2^\times$ have
  $x(P) \equiv 1 - h(1)^2 \bmod 8$, and such points exist.
  If $\delta_2(i(P)) = x(P) - \theta$
  is not in the image of~$\Sel_2(J)$ under~$\rho_2$ for one of them, then $C(\Q)$
  does not contain points~$P$ such that $x(P) \in \Z_2^\times$.
\end{Lemma}

\begin{proof}
  We first note that $f'(a) = (2g+1) a^{2g} + 2 h(a) h'(a) \in \Z_2^\times$ when
  $a \in \Z_2^\times$, which implies that $f(a + 2^m k) \equiv f(a) + 2^m k \bmod 2^{m+1}$
  for $k \in \Z_2$.

  Since $h(1)$ is even by assumption, we have $f(1) = 1 + h(1)^2 \in 1 + 4\Z$.
  Then $f(1 - h(1)^2) \equiv 1 \bmod 8$,
  hence there is $P \in C(\Q_2)$ with $x(P) = 1 - h(1)^2$.

  If $P' \in C(\Q_2)$ is any other point such that $x(P') \in \Z_2^\times$,
  write $x(P') = x(P) + 2k$ with $k \in \Z_2$. Then $f(x(P')) \equiv 1 + 2k \bmod 4$,
  so $k$ must be even for $f(x(P'))$ to be a square. We can then write
  $x(P') = x(P) + 4k$ with $k \in \Z_2$. Then $f(x(P')) \equiv 1 + 4k \bmod 8$,
  so $k$ must again be even, and we have that $x(P') \equiv x(P) \bmod 8$.

  Write $f = f_1 \cdots f_m$ with the $f_j$ the monic irreducible factors
  of~$f$ in~$\Q_2[x]$. Then $\Q_2[\theta]$ is isomorphic to the product of
  the $m$ (unramified) extensions of~$\Q_2$ obtained by adjoining a root of $f_1, \ldots, f_m$.
  Since $f(1)$ and $x(P)$ are odd, $f_j(x(P)) \in \Z_2^\times$ for all~$j$,
  hence $x(P) - \theta$ is a $2$-adic unit in each component of this product.
  Then $x(P') \equiv x(P) \bmod 8$ implies that the square classes of
  $x(P') - \theta$ and $x(P) - \theta$ are the same, and so $\delta_2(i(P')) = \delta_2(i(P))$.
  Since by assumption, $\delta_2(i(P)) \notin \rho_2(\Sel_2(J))$, none of these points
  can be in~$C(\Q)$.
\end{proof}

\begin{Remark}
  If $h(1)$ is odd (hence $f(1)$ is even), dealing with points~$P$
  such that $x(P) \in \Z_2^\times$ can be more involved, depending
  on the precise behavior of~$h$. We do not discuss this here.
\end{Remark}

It remains to consider points whose $x$-coordinate has either strictly positive
or strictly negative $2$-adic valuation. We consider rational points~$P$
with $v_2(x(P)) > 0$ first. We can assume that $P$ is $2$-adically closer
to~$P_0$ than to~$\iota(P_0)$. Then $\log' i(P) = (\int_{P_0}^P \bfomega) \cdot U$,
and we can express the right hand side as a vector of power series in
the uniformizer $t = x$ at~$P_0$. So we need to determine the set
\[ Z(P_0) := \{\rho \log' i(\tau, *) : 0 \neq \tau \in 2 \Z_2\} \subset \F_2^g \,, \]
where the $y$-coordinate is chosen so that the point is close to~$P_0$.
This is a finite
problem, since $\rho$ is locally constant and $\rho \log' i(\tau, *)$
is given by the reduction of the coefficient of~$t$ in the series vector
whenever the valuation of~$\tau$ is large enough for the first term in
the series to dominate all others. We then check that $Z(P_0)$ has empty
intersection with the image of the Selmer group. (See also~\cite{PoonenStoll2014},
where this kind of argument is used to show that most odd degree
hyperelliptic curves have the point at infinity as their only rational point.)

When $v_2(x(P)) < 0$, then $P$ is in the residue disk of~$\infty$.
We express $\log' i(P) = (\int_\infty^P \bfomega) \cdot U$ as power
series in $t = y/x^{g+1}$, which is a uniformizer at~$\infty$.
Since $\infty$ is a Weierstrass point and $t$ is odd with respect to
the hyperelliptic involution, the series
giving~$\bfomega$ involve only even powers of~$t$, so no loss of precision
is introduced by the formal integral. We then compute $Z(\infty)$ in an
analogous way as we did~$Z(P_0)$ and check that it does not meet the image
of the Selmer group. We obtain the following general statement.

\begin{Theorem} \label{T:method}
  Assume {\upshape (A\ref{ass1})--(A\ref{ass3})} and that $h(1)$ is even.
  If $1 - h(1)^2 - \theta \notin \rho_2(\Sel_2(J))$
  and neither $Z(P_0)$ nor $Z(\infty)$ meets $\lbar\bigl(\rho_2(\Sel_2(J))\bigr)$,
  and $\log' i(P) \neq 0$ for all
  $P \in C(\Q_2)$ with $v_2(x(P)) \neq 0$ other than $P_0$, $\iota(P_0)$ and~$\infty$,
  then $C(\Q) = \{\infty, P_0, \iota(P_0)\}$.
\end{Theorem}

We have implemented the procedure described above in Magma~\cite{Magma}.
This implementation is available at~\cite{programs}*{\texttt{SelChabDyn/}}.

%=============================================================================

\section{An application to arithmetic dynamics} \label{S:dyn}

As an application of the method described above, we show that the curve
\begin{align*}
  C \colon y^2 = a_5(x) &=
         x^{15} + x^{14} + 2 x^{13} + 5 x^{12} + 14 x^{11} + 26 x^{10} + 44 x^9 + 69 x^8 \\
    & \qquad {} + 94 x^7 + 114 x^6 + 116 x^5 + 94 x^4 + 60 x^3 + 28 x^2 + 8 x + 1
\end{align*}
of genus~$7$ has exactly the three rational points $\infty$, $(0,1)$ and~$(0,-1)$.
As usual, we write $J$ for its Jacobian variety. Since $a_5(x) = x^{15} + a_4(x)^2$
with $\deg a_4 = 7$, $a_4(0) = 1$ and $a_4(1) = 26$, Assumptions~\ref{A:Ass1}
and the assumption in Lemma~\ref{L:v=0}
are satisfied. We can also check that $a_5$ is irreducible and that
\[ \disc(a_5) = 13 \cdot 24554691821639909 \]
is squarefree. Let $\theta$ be a root of~$a_5$ and let $K = \Q(\theta)$.
Then the ring of integers of~$K$ is $\CO_K = \Z[\theta]$. One checks (using
Magma~\cite{Magma}, say) that $\CO_K$ has trivial narrow class group (the Minkowski
bound is less than~$10^4$, so this can easily be done unconditionally).
This implies that both the usual class group and the narrow class group
of~$\CO_K[\frac{1}{2}]$ are trivial, as they both are quotients of the
narrow class group of~$\CO_K$. From this one can deduce that $\rho_2$ is injective.

Over~$\F_2$, $a_5$ splits into a product of three irreducible factors of degree~$5$.
So $A \otimes \Q_2$ is a product of three copies of the unramified extension
of~$\Q_2$ of degree~$5$. In particular, $\dim J(\Q_2)[2] = 2$.
We set up~$H_2$ and compute the image of~$J(\Q_2)$ under~$\delta_2$
using Lemma~\ref{L:imlog}. The set~$I$ can be taken to be~$\{3, 5\}$.
We intersect this image with the image of~$\widetilde{\Sel}$, which is the
subgroup of $\CO_K^\times/\CO_K^{\times 2}$ consisting of the classes of units
with square norm
whose images under the real embeddings corresponding to the two largest
roots of~$a_5$ have the same sign.
This intersection, which is isomorphic to the $2$-Selmer group, has
dimension~$2$. This latter fact was already observed in~\cite{DHJMS}.
We will identify the Selmer group with its image in~$H_2$.
Since there is no rational $2$-torsion, this implies that the rank
of~$J(\Q)$ is at most~$2$. However, we were unable to find any point
of infinite order in~$J(\Q)$.
So we cannot use the standard Chabauty method to determine~$C(\Q)$.

Points $P \in C(\Q_2)$ with $x(P) \in \Z_2^\times$ have $x(P) \equiv -3 \bmod 8$.
We check that the image of $-3-\theta$ in~$H_2$ is not in the Selmer group.
This already shows that there are no rational points on~$C$ whose $x$-coordinate
is a $2$-adic unit; see Lemma~\ref{L:v=0}.

We then compute the image of~$\log$, again using Lemma~\ref{L:imlog}.
The transformation matrix~$U$ can be taken to be
\[ U = \begin{pmatrix}
         1/2 &   0 &   0 &   0 &   0 &    0 &    0 \\
           0 & 1/2 &   0 &   0 &   0 &    0 &    0 \\
           0 &   0 & 1/2 &   0 &   0 & -1/4 &    0 \\
           0 &   0 &   0 & 1/2 &   0 &    0 &    0 \\
           0 &   0 &   0 &   0 & 1/2 &    0 & -1/4 \\
           0 &   0 &   0 &   0 &   0 &  1/4 &    0 \\
           0 &   0 &   0 &   0 &   0 &    0 &  1/4
       \end{pmatrix} .
\]
In terms of the corresponding basis of $\F_2^g = (\im \log')/2(\im \log')$,
the image of the Selmer group is generated by
$(0, 1, 0, 0, 0, 1, 0)$ and $(0, 0, 1, 1, 1, 1, 0)$. This shows
that $\lbar \circ \rho_2$ is injective, so~(A\ref{ass3}) holds.

Then $\log' i(2t, *)$ is given by the vector of power series
\begin{align*}
  \bigl(&t - 2^2 t^2 + 2^2 t^4 + O(2^3),\quad t^2 + 2^2 t^4 + O(2^3), \quad 2^2 t^3 + O(2^3), \\
        &2 t^4 + O(2^3), \quad O(2^3), \quad 5 \cdot 2 t^3 + 2^2 t^4 + O(2^3), \quad O(2^3)\bigr) \,.
\end{align*}
For odd~$t$, this gives $\rho \log' i(2t,*) = (1, 1, 0, 0, 0, 0, 0)$.
For even~$t$, the linear term is dominant and gives $\rho \log' i(2t,*) = (1, 0, 0, 0, 0, 0, 0)$.
Both these vectors are not in the image of the Selmer group.

For points near~$\infty$, we write $P(2t)$ for the point with $y/x^8 = 2t$.
Then the corresponding series are
\[ \bigl(O(2^2), \; O(2^2), \; O(2^2), \; O(2^2), \; O(2^2), \; O(2^2), \; -t + O(2^2)\bigr) \,. \]
So $\rho \log' i(P(2t)) = (0, 0, 0, 0, 0, 0, 1)$ for all $0 \neq t \in \Z_2$,
which is also not in the image of the Selmer group.

We see that the method is successful and shows that there are no unexpected rational
points on~$C$. Since $A_5(x) = x^{16} a_5(1/x)$, this proves the following.

\begin{Proposition} \label{P:a5}
  The only rational number~$c$ such that $A_5(c)$ is a square in~$\Q$ is $c=0$.
\end{Proposition}

We now discuss how to deal with $A_n(x) = y^2$ for composite $n \ge 6$.
We begin with an auxiliary result.

\begin{Lemma} \label{L:rigdiv}
  There is a sequence $(B_n)_{n \ge 1}$ of monic polynomials $B_n \in \Z[c]$
  such that $A_n(c) = \prod_{d \mid n} B_d(c)$ for all $n \ge 1$. If $m \neq n$,
  then the resultant of $B_m$ and~$B_n$ is~$\pm 1$.
\end{Lemma}

\begin{proof}
  (See also~\cite{BEK2020}*{Lemma~9.5}, which the authors
  attribute to Bjorn Poonen. The polynomials~$A_n$ are known as Gleason polynomials;
  the complex roots of the~$B_n$ are the centers of period exactly~$n$ in the Mandelbrot set.)
  We define $B_n = \prod_{d \mid n} A_d^{\mu(n/d)} \in \Q(c)$, where $\mu$ is
  the M\"obius function. Then the relation between the $A_n$'s and
  the~$B_n$'s is satisfied. The proof of Lemma~1.1(b) in~\cite{Stoll1992},
  which gives a similar result for the values $A_n(c)$ when $c$ is an integer,
  works in the same way over~$\F_p[c]$ in place of~$\Z$ (it uses that we
  have a PID with finite residue fields). It shows that the reduction mod~$p$
  of~$B_n$ is a (monic) polynomial and that the reductions of $B_m$ and~$B_n$
  are coprime when $m \neq n$. Since this holds for every prime~$p$, the
  first shows that $B_n \in \Z[c]$. The second means that the
  resultant is not divisible by~$p$ for any~$p$, so
  the resultant must be a unit in~$\Z$.
\end{proof}

From the recurrence defining~$A_n$, we find that
\[ A_n(0) = 0\,, \qquad
   A_n(-1) = \begin{cases} -1, & \text{$n$ odd} \\ \hfill 0, & \text{$n$ even} \end{cases}
   \quad\text{and}\quad
   A_n(-2) = \begin{cases} -2, & n = 1\,, \\ \hfill 2, & n > 1\,. \end{cases}
\]
We also find that $A'_n(0) = 1$, $A'_n(-1) = (-1)^{n-1}$. This implies that
\[ B_n(0) = \begin{cases} 0, & n = 1, \\ 1, & n > 1, \end{cases} \quad\!
   B_n(-1) = \begin{cases} -1, & n = 1, \\
                     \hfill 0, & n = 2, \\
                     \hfill 1, & n > 2,
             \end{cases} \quad\!
   B_n(-2) = \begin{cases} -2, & n = 1, \\
                           -1, & \text{$n > 1$ squarefree,} \\
                     \hfill 1, & \text{else.}
             \end{cases}
\]

\begin{Lemma} \label{L:AmAn}
  Assume that $m \ge 2$ divides~$n$ and that $m$ is even when $n$ is even.
  If $c \in \Q$ is such that $A_n(c)$ is a square in~$\Q$, then $A_m(c)$ is
  also a square in~$\Q$.
\end{Lemma}

\begin{proof}
  We write $A_n = A_m B$ where $B$ is the product of the~$B_d$ with $d$
  a divisor of~$n$ that is not a divisor of~$m$. Our assumptions imply
  that all these $d$ are $> 2$. This shows that $B(0) = B(-1) = 1$
  and also $B(-2) = A_n(-2)/A_m(-2) = 1$. In addition, $B$ is monic of even degree.
  Together, these imply that $B(c)$ is a nonzero $3$-adic square.

  When $A_m(c) = 0$, the claim is trivially true. Otherwise,
  since the resultant of $A_m$ and~$B$ is~$\pm 1$ by Lemma~\ref{L:rigdiv}
  and the multiplicativity of the resultant, and since both $A_m$ and~$B$ have even degree,
  either both $A_m(c)$ and~$B(c)$ are squares or both their negatives are squares.
  Since $B(c)$ is a nonzero $3$-adic square, its negative is a non-square
  and it follows that $B(c)$ and~$A_m(c)$ must both be squares.
\end{proof}

\begin{Corollary} \label{C:reduce} \strut
  \begin{enumerate}[\upshape(1)]\addtolength{\itemsep}{6pt}
    \item Assume that $n$ is odd, $p$ is a prime divisor of~$n$ and that
          $c \in \Q$ is such that $A_n(c)$ is a square. Then $A_p(c)$ is also a square.
    \item Assume that $4 \mid n$. Then the only $c \in \Q$ such that $A_n(c)$ is a square
          are $c = 0$ and $c = -1$.
    \item Assume that $n = 2m$ with $m$ odd. If $A_n(c)$ is a square and $p$
          is a prime divisor of~$m$, then $A_{p}(c)$ is a square or $-A_p(c)$ is a square.
          Also, $A_2(c) = c(c+1)$ is a square.
  \end{enumerate}
\end{Corollary}

\begin{proof} \strut
  \begin{enumerate}[\upshape(1)]\addtolength{\itemsep}{6pt}
    \item Take $m = p$ in Lemma~\ref{L:AmAn}.
    \item Take $m = 4$ in Lemma~\ref{L:AmAn} and use that $A_4(c)$ is a
          square only for $c = 0, -1$; see~\cite{DHJMS}.
    \item The first statement follows as in the proof of Lemma~\ref{L:AmAn};
          the difference is that in this case, we cannot show that $B(c)$ must
          be a square. For the second statement, take $m = 2$ in Lemma~\ref{L:AmAn}.
    \qedhere
  \end{enumerate}
\end{proof}

Since each of $A_3(c)$ and $A_5(c)$ is a square only for $c = 0$ and $A_4(c)$
is a square only for $c = 0$ and $c = -1$, this already implies that
$A_n(c)$ cannot be a square when $c \neq 0, -1$ and $n$ is an odd multiple
of $3$ or~$5$, or a multiple of~$4$. We now consider~$A_6$.

\begin{Proposition}
  Let $c \in \Q$ be such that $A_6(c)$ is a square. Then $c = 0$ or $c = -1$.
\end{Proposition}

\begin{proof}
  By the corollary above, $A_2(c)$ is a square. This implies
  that $c = 0$ or $c = 1/(t^2-1)$ for some $t \in \Q \setminus \{\pm 1\}$. We also
  know that $c = 0$ when $A_3(c)$ is a square.
  So we can assume that $-A_3(c)$ is a square. Substituting
  $c = 1/(t^2-1)$ into $-A_3(c) = y^2$ and writing $u = y (t^2-1)^2$,
  we obtain the equation
  \[ u^2 = -t^6 + 2 t^4 - 3 t^2 + 1 \,; \]
  this defines a curve of genus~$2$ over~$\Q$.
  A $2$-descent on its Jacobian as implemented in Magma and described
  in~\cite{Stoll2001} shows that its Mordell-Weil rank is at most~$1$.
  The difference of the two rational points on the curve with $t$-coordinate
  zero is a point of infinite order on the Jacobian. Applying Magma's
  \texttt{Chabauty} function (based on~\cite{BruinStoll2010}*{Section~4.4})
  to it shows that these two are the only rational points on the curve.
  They correspond to $c = -1$; this completes the proof. (The fairly new
  Magma function \texttt{RationalPointsGenus2} takes about~0.1 seconds
  to obtain this result. See~\cite{Stoll2026} for a description of how it works.)
\end{proof}

This implies that $A_n(c)$ is never a square for $c \neq 0, -1$
when $n$ is any multiple of~$3$.

We now consider $A_7$. We will work with the curve $y^2 = a_7(x)$ instead,
which is isomorphic to $y^2 = A_7(x)$. When $n$ is odd, $a_{n-1}(0) = 1$
and $a_{n-1}(-1) = 0$ (and so $a_{n-1}(1) \equiv a_{n-1}(-1) = 0 \bmod 2$),
so the assumptions that $h(0)$ is odd and $h(1)$ is
even are always satisfied for $f(x) = a_n(x) = x^{2g+1} + a_{n-1}^2$
with $g = 2^{n-2}-1$. The discriminant of~$a_7$ is squarefree,
\begin{align*}
  \disc(a_7) &= -8291 \cdot 9137 \cdot 420221 \cdot 189946395389 \cdot 4813162343551332730513 \\
             &{}\qquad \cdot 2837919018511214750008829 \\
             &{}\qquad \cdot 1858730157152877176856713108209153714699601\,,
\end{align*}
and $a_7$ is irreducible. Let $K = \Q(\theta)$ with $\theta$ a root of~$a_7$.
Assuming GRH, we can verify with Magma that $K$ has trivial class group.
(The Minkowski bound is way too large to make an unconditional computation
of the class group feasible. Under GRH, it takes about two and a half hours
on the author's current laptop.)
So the $2$-Selmer group is contained in the units modulo squares. We can then get
generators of the unit group quickly as power products on a factor base of
elements of~$K$ (using Magma's \texttt{SUnitGroup} with the optional parameter
\texttt{Raw}). This allows us to compute the signs of these units in the nine
different real embeddings of~$K$ and thus to find the subgroup of the group
of units modulo squares consisting of elements whose sign vectors are in
the image of~$J(\R)$ under the local descent map~$\delta_\infty$. This subgroup~$U_\infty$
has dimension $35 + 1 - 5 = 31$ as an $\F_2$-vector space (the unit rank
is~$35$, as there are $9$ real embeddings and $27$~pairs of complex embeddings
of~$K$; we add~$1$ for the torsion subgroup, and the local conditions at
infinity cut out a subspace of codimension~$5$, as the computation shows).
We then compute the image
of~$\delta_2$ as in Lemma~\ref{L:imlog} and the map $\rho_2 \colon U_\infty \to H_2$
and find that $\rho_2$ is injective and that its image intersects the
image of~$\delta_2$ trivially. This means that the $2$-Selmer group, which
is isomorphic to this intersection, is trivial, so the three obvious points
are the only rational points on $y^2 = a_7(x)$. This implies the following.

\begin{Proposition}
  If $K$ as above has odd class number (which is true assuming GRH),
  then the only rational number~$c$ such that $A_7(c)$ is a square in~$\Q$ is $c=0$.
\end{Proposition}

We can also deal with $n = 10$.

\begin{Proposition} \label{P:A_10}
  Let $c \in \Q$ be such that $A_{10}(c)$ is a square. Then $c = 0$ or $c = -1$.
\end{Proposition}

\begin{proof}
  Assume that $A_{10}(c)$ is a square. By Corollary~\ref{C:reduce}, it follows
  that $A_5(c)$ is a square or $-A_5(c)$ is a square. In the first case,
  $c = 0$ by Proposition~\ref{P:a5}. So we consider the second case. If $-A_5(c)$
  is a square, then there is a rational point with $x$-coordinate~$-1/c$ on
  the curve
  \[ C_- \colon y^2 = -a_5(-x) = x^{15} - a_4(-x)^2 \,. \]
  This curve is isomorphic to the quadratic twist by~$-1$ of the curve~$C$
  discussed earlier; in particular, $C$ and~$C_-$ are isomorphic over~$\Q(i)$.
  We use this isomorphism for the computation of logarithms.

  First, we compute the $2$-Selmer group of the Jacobian~$J_-$ of~$C_-$, which
  turns out to have dimension~$3$. By searching, we find
  that the points in~$J_-(\Q_2)$ whose $a$-polynomials in the Mumford representation
  are in the list given below give a basis of~$J_-(\Q_2)/2J_-(\Q_2)$.
  \begin{gather*}
    x^2 + 2 x + 2\,, \quad x^4 + 2 x^3 + 2\,, \quad x^4 + 2 x^2 + 2\,, \quad
    x^6 + 2 x^3 + 2\,, \quad x^6 + 2 x^5 + 2 x^4 + 2 x^3 + 2\,, \\
    x - 1\,, \quad x^2 + x + 1\,, \quad x^3 - 2 x^2 + 3 x - 1\,, \quad
    x^6 + 2 x^5 - 2 x^3 - 3 x - 3
  \end{gather*}
  The first five of these have roots~$\alpha$ satisfying $v_2(\alpha) \ge 1/6$.
  Their logarithms can be computed using the expansion around~$(0, i)$
  in a similar way as for~$C$. (The polynomials split into two factors
  over~$\Q_2(i)$ corresponding to the two branches $y \sim \pm i h(x)$
  near $x = 0$; one has to take the difference of the logarithms computed
  for the two factors to make sure the same base-point~$(0, i)$ is used.
  Note that the image of this point under the usual embedding is torsion.)
  The remaining four have roots that are $2$-adic units. They fall into
  the domain of convergence of the logarithm expansions around~$\infty$
  in terms of $t = (y + i h(x))/x^8$ (with $h(x) = a_4(-x)$).
  As a consistency check, the logarithms we obtain are indeed in~$\Q_2$
  (to the target precision), even though the computation involves~$\Q_2(i)$
  in a nontrivial way. We then proceed as for~$C$. We find that the image
  of the Selmer group under $\lbar \circ \rho_2$ has dimension~$3$, which
  verifies~(A\ref{ass3}).

  We note that there are
  no rational (not even $2$-adic) points~$P$ on~$C_-$ with $v_2(x(P)) > 0$,
  since for even~$x$, we have that $-a_5(-x) \equiv -1 \bmod 8$.
  To deal with points that have $v_2(x(P)) = 0$ and $y(P) \equiv 1 \bmod 4$,
  we use the base point~$(1,1)$: we expand the logarithms of $[P - (1,1)]$
  at~$(1,1)$ in terms of $t = x-1$. It is easy to check that $v_2(t)$
  must be at least~$3$ to give a $\Q_2$-point on the curve. For such~$t$,
  $\rho(\log' [P - (1,1)])$ is always the same vector, which is not in the image
  of the Selmer group. By Corollary~\ref{C:disc} with $B = (1,1)$
  (and similarly, with $B = (1,-1)$ for the points with $y(P) \equiv -1 \bmod 4$),
  $(1, \pm 1)$ are the only rational points with~$v_2(x(P)) = 0$.
  Note that $[(1,1) - \infty]$ has infinite order here and in particular
  non-vanishing logarithm, so just looking at $\rho(\log' i(P))$ would not
  give a contradiction (the value is that for~$B$, which is in the image
  of the Selmer group, coming from a rational point).

  For points with $v_2(x(P)) < 0$, we proceed in the
  same way as for~$C$; again, the only image under~$\rho$ that occurs
  is not in the image of the Selmer group. We conclude that
  $C_-(\Q) = \{\infty, (1,1), (1,-1)\}$, which means that $c \in \{-1,0\}$.
\end{proof}

We conclude that for $c \in \Q \setminus \{0, -1\}$, $A_n(c)$ is not a
square in~$\Q$ when $n$ is a multiple of $3$, $4$ or~$5$ or (assuming GRH)
an odd multiple of~$7$. This implies that when $f^{\circ 2}_c$ is irreducible,
then so is $f^{\circ 6}_c$, and under GRH, $f^{\circ 10}_c$ is irreducible
as well. This proves Theorem~\ref{T:appl}.

We remark that it appears to be hopeless to use the method presented here
to show that $A_{11}(c)$ is never a square for $0 \neq c \in \Q$, as this
would require showing that the class number of the number field of degree~$1023$
given by adjoining a root of~$a_{11}$ (which is irreducible) to~$\Q$ is odd.

A Magma program verifying the computational claims of this section is
available at the author's GitHub site~\cite{programs}*{\texttt{SelChabDyn/}}.
It uses the implementation of the algorithm described in Section~\ref{S:method}.

%=============================================================================

\section{Further examples} \label{S:examples}

We take $h(x) = x + 1$ and consider
\[ C_g \colon y^2 = x^{2g+1} + (x+1)^2 \]
for $g \ge 1$. Then Assumptions~\ref{A:Ass1} are satisfied.
We ran our algorithm for $g = 1, 2,\ldots, 12$.
In each case, (A\ref{ass3}) can be checked to be satisfied
(and $h(1) = 2$ is even), so that our
method can be used. The following table summarizes the results;
the entries in the second row give the $\F_2$-dimension of the $2$-Selmer group,
and a $+$ or~$-$ in the third row says whether the method was successful or not
in showing that $C_g(\Q) = C_g(\Q)_{\odd}$. (We assume GRH for the class
group computations when $g \ge 11$.)

\[ \renewcommand{\arraystretch}{1.2}
  \begin{array}{|r||*{12}{c|}} \hline
                   g & 1 & 2 & 3 & 4 & 5 & 6 & 7 & 8 & 9 & 10 & 11 & 12 \\\hline
           \dim \Sel & 0 & 1 & 0 & 1 & 0 & 2 & 1 & 1 & 2 &  1 &  0 &  2 \\\hline
     \text{success?} & + & - & + & + & + & + & + & + & + &  + &  + &  + \\\hline
  \end{array}
\]

For $g = 2$, we are not successful, and indeed, $C_2(\Q)$ contains further
points with $x$-coordinate~$12$. Since the Selmer group is one-dimensional,
the Mordell-Weil rank is~$1$, with $[(12, 499) - \infty]$ a rational point
of infinite order on the Jacobian. A combination of Chabauty's method with
the Mordell-Weil Sieve, as implemented in Magma, allows us to determine~$C_2(\Q)$;
see~\cite{BruinStoll2010}*{Section~4.4} and~\cite{Stoll2026}.
We obtain the following result.

\begin{Proposition}
  Let $g \in \{1,2,\ldots,12\}$ and let $C_g \colon y^2 = x^{2g+1} + (x+1)^2$.
  If $g \neq 2$, then $C_g(\Q) = \{\infty, (0,1), (0,-1)\}$ (for $g \ge 11$,
  this assumes GRH), whereas
  \[ C_2(\Q) = \{\infty, (0,1), (0,-1), (12, 499), (12, -499)\} \,. \]
\end{Proposition}

We note that $C_1$ is the curve $y^2 = a_3(x)$, where $a_n$ is defined as
in the preceding section.

We have also run the algorithm systematically for several values of the genus~$g$
with all polynomials $h \in \Z[x]$ of degree~$\le g$, with coefficients bounded in
absolute value by a quantity depending on~$g$ and such that $h(0)$ is odd and $h(1)$
is even. We can restrict to polynomials with positive leading coefficient.
We did the computations assuming GRH to speed up the determination
of the class groups. The results are summarized in the tables below.

The curves are partitioned according to the dimension of the $2$-Selmer group
and sorted into three buckets, according to whether our algorithm was successful
in proving that there are only the three obvious rational points, the curve
has more than these three rational points (up to a naive $x$-coordinate height
of $4 \log 10$, extended to $7 \log 10$ for the curves where the algorithm was unsuccessful),
or the algorithm was unsuccessful, but no additional points
were found. We also give the average size of the Selmer group. We note that
the total average is reasonably close to~$3$, which is consistent with the
result of~\cite{BG2013}, even though their result does not apply to our restricted
family of curves.

For curves with an additional pair of points, one would expect an average of~$6$
(for all odd degree curves of fixed genus with a marked non-Weierstrass point;
see~\cite{Shankar2019}*{Theorem~1.1}, which shows it is at most~$6$, and
equality is expected, see p.~270 there),
which is consistent with the trend as $g$ grows in our tables. It is to be expected,
and also visible in the data, that the algorithm is more likely to be successful
when the Selmer group is small. The increasing success rate as the genus grows
is consistent with the result of~\cite{PoonenStoll2014}*{Theorem~10.6}, which says that
for not too small genus, the method will be successful most of the time in
showing that a general hyperelliptic curve of odd degree has the point at
infinity  as its only rational point, with the probability of failure bounded by
$(12 g + 20) 2^{-g}$ in that larger family. We note that our assumption~(A\ref{ass3})
is exactly the condition in~\cite{PoonenStoll2014}*{Proposition~6.2} when $J(\Q)[2] = 0$
(which is the case with probability~$1$).

\bigskip

\textbf{Genus~2}, coefficients bounded by~$25$.
\[ \renewcommand{\arraystretch}{1.2}
  \begin{array}{|r||*{6}{c|}|c||c|c|} \hline
           \dim \Sel & 0 & 1 & 2 & 3 & 4 & 5 & \text{avg.~$\#\Sel$}
                     & \text{total} & \text{perc.} \\ \hline \hline
              \text{success} & 4208 & 18 & 0 & 0 & 0 & 0 & 1.004 & 4226 & 25.0\% \\ \hline
            \text{more pts.} & 0 & 1138 & 1390 & 469 & 36 & 2 & 4.029 & 3035 & 18.0\% \\ \hline
              \text{failure} & 0 & 6325 & 2882 & 414 & 18 & 0 & 2.882 & 9639 & 57.0\% \\ \hline\hline
                \text{total} & 4208 & 7481 & 4272 & 883 & 54 & 2 & 2.618 & 16900 & 100.0\% \\ \hline
  \end{array}
\]

\bigskip

\textbf{Genus~3}, coefficients bounded by~$9$.
\[ \renewcommand{\arraystretch}{1.2}
  \begin{array}{|r||*{6}{c|}|c||c|c|} \hline
           \dim \Sel & 0 & 1 & 2 & 3 & 4 & 5 & \text{avg.~$\#\Sel$}
                     & \text{total} & \text{perc.} \\ \hline \hline
              \text{success} & 4004 & 2963 & 161 & 0 & 0 & 0 & 1.483 & 7128 & 41.6\% \\ \hline
            \text{more pts.} & 0 & 760 & 1150 & 569 & 59 & 2 & 4.598 & 2540 & 14.8\% \\ \hline
              \text{failure} & 0 & 3795 & 3093 & 556 & 37 & 1 & 3.346 & 7482 & 43.6\% \\ \hline\hline
                \text{total} & 4004 & 7518 & 4404 & 1125 & 96 & 3 & 2.757 & 17150 & 100.0\% \\ \hline
  \end{array}
\]

\bigskip

\textbf{Genus~4}, coefficients bounded by~$5$.
\[ \renewcommand{\arraystretch}{1.2}
  \begin{array}{|r||*{6}{c|}|c||c|c|} \hline
           \dim \Sel & 0 & 1 & 2 & 3 & 4 & 5 & \text{avg.~$\#\Sel$}
                     & \text{total} & \text{perc.} \\ \hline \hline
              \text{success} & 4806 & 6317 & 1715 & 38 & 0 & 0 & 1.911 & 12876 & 58.6\% \\ \hline
            \text{more pts.} & 0 & 731 & 1218 & 666 & 120 & 9 & 5.055 & 2744 & 12.5\% \\ \hline
              \text{failure} & 0 & 2381 & 2969 & 898 & 91 & 1 & 3.992 & 6340 & 28.9\% \\ \hline\hline
                \text{total} & 4806 & 9429 & 5902 & 1602 & 211 & 10 & 2.905 & 21960 & 100.0\% \\ \hline
  \end{array}
\]

\bigskip

\textbf{Genus~5}, coefficients bounded by~$3$.
\[ \renewcommand{\arraystretch}{1.2}
  \begin{array}{|r||*{6}{c|}|c||c|c|} \hline
           \dim \Sel & 0 & 1 & 2 & 3 & 4 & 5 & \text{avg.~$\#\Sel$}
                     & \text{total} & \text{perc.} \\ \hline \hline
              \text{success} & 3307 & 5708 & 2447 & 282 & 5 & 0 & 2.285 & 11749 & 69.9\% \\ \hline
            \text{more pts.} & 0 & 693 & 1204 & 644 & 133 & 6 & 5.102 & 2680 & 15.9\% \\ \hline
              \text{failure} & 0 & 746 & 1110 & 463 & 59 & 1 & 4.461 & 2379 & 14.2\% \\ \hline\hline
                \text{total} & 3307 & 7147 & 4761 & 1389 & 197 & 7 & 3.042 & 16808 & 100.0\% \\ \hline
  \end{array}
\]

\bigskip

\textbf{Genus~6}, coefficients bounded by~$3$.
\[ \renewcommand{\arraystretch}{1.2} \setlength{\arraycolsep}{4pt}
  \begin{array}{|r||*{7}{c|}|c||c|c|} \hline
           \dim \Sel & 0 & 1 & 2 & 3 & 4 & 5 & 6 & \text{avg.~$\#\Sel$}
                     & \text{total} & \text{perc.} \\ \hline \hline
              \text{success} & 22093 & 41016 & 22072 & 4096 & 204 & 4 & 0 & 2.554 & 89485 & 76.1\% \\ \hline
            \text{more pts.} & 0 & 3783 & 7629 & 4539 & 1116 & 109 & 7 & 5.598 & 17183 & 14.6\% \\ \hline
              \text{failure} & 0 & 2792 & 4990 & 2673 & 497 & 28 & 0 & 5.080 & 10980 & 9.3\% \\ \hline\hline
                \text{total} & 22093 & 47591 & 34691 & 11308 & 1817 & 141 & 7 & 3.235 & 117648 & 100.0\% \\ \hline
  \end{array}
\]

\bigskip

\textbf{Genus~7}, $1000$ random polynomials~$h$ with coefficients bounded by~$10$.
\[ \renewcommand{\arraystretch}{1.2}
  \begin{array}{|r||*{6}{c|}|c||c|c|} \hline
           \dim \Sel & 0 & 1 & 2 & 3 & 4 & 5 & \text{avg.~$\#\Sel$}
                     & \text{total} & \text{perc.} \\ \hline \hline
              \text{success} & 199 & 398 & 222 & 56 & 4 & 0 & 2.725 & 879 & 87.9\% \\ \hline
            \text{more pts.} & 0 & 11 & 20 & 14 & 4 & 0 & 5.673 & 49 & 4.9\% \\ \hline
              \text{failure} & 0 & 19 & 26 & 19 & 7 & 1 & 6.083 & 72 & 7.2\% \\ \hline\hline
                \text{total} & 199 & 428 & 268 & 89 & 15 & 1 & 3.111 & 1000 & 100.0\% \\ \hline
  \end{array}
\]

%=============================================================================

\end{document}